\documentclass[11pt]{amsart}

\usepackage{amsmath}
\usepackage{amssymb}
\usepackage{amsfonts}
\usepackage{euscript}

\newcommand{\Z}{\mathbb{Z}} 
 \newcommand{\Q}{\mathbb{Q}}

\newcommand{\EE}{\EuScript {E}}
\newcommand{\OO}{\EuScript {O}}

\newcommand{\spl}{\mathrm{SL}}
\newcommand{\stab}{\mathrm{Stab}}

\newtheorem{thm}{Theorem}[section]
\newtheorem{lem}[thm]{Lemma}
\newtheorem{prop}[thm]{Proposition}

\theoremstyle{definition}
\newtheorem{defn}[thm]{Definition}
\newtheorem{rem}[thm]{Remark}
\renewcommand{\vee}{\ast}

\begin{document} 

\vspace*{-0.5cm}

\title[binary cubic forms II]
{Relations among Dirichlet series whose coefficients are
class numbers of binary cubic forms II}
\date{July 7, 2011}

\begin{abstract}
As a continuation of the authors and Wakatsuki's
previous paper \cite{sty}, we study relations among Dirichlet
series whose coefficients are class numbers of binary
cubic forms.
We show that for any integral models of the space of
binary cubic forms, the associated
Dirichlet series satisfies a simple explicit relation
to that of the dual other than the usual functional equation.
As an application, we write the
functional equations of these Dirichlet series
in self dual forms.
\end{abstract}

\author[Yasuo Ohno]{Yasuo Ohno}
\address{
Department of Mathematics, Kinki University,
Kowakae 3-4-1, Higashi-Osaka, Osaka 577-8502, Japan}
\email{ohno@math.kindai.ac.jp}
\author[Takashi Taniguchi]{Takashi Taniguchi}
\address{
Department of Mathematics, Graduate School of Science, Kobe University,
1-1, Rokkodai, Nada-ku, Kobe 657-8501, Japan}
\address{
Department of Mathematics, Princeton University,
Fine Hall, Washington Road, Princeton, NJ 08540}
\email{tani@math.kobe-u.ac.jp}
\maketitle

\vspace*{-0.5cm}

\section{Introduction}\label{sec:intro}

The theory of zeta functions for the space of binary cubic forms
was initiated by Shintani \cite{shintania} as a fine example
of zeta functions of prehomogeneous vector spaces \cite{sash}.
He introduced $4$ Dirichlet series
$\xi_{1,1}(s)$, $\xi_{1,2}(s)$, $\xi^\vee_{1,1}(s)$,
$\xi^\vee_{1,2}(s)$ whose coefficients are class numbers of
integral binary cubic forms, and established their remarkable beautiful
analytic properties. These $4$ zeta functions he introduced are
for the ``standard'' integral models,
and our purpose is to study the zeta functions
for {\em all} integral models.

Let us recall the definition of the zeta function.
Let $V_\Q$ be the space of binary cubic forms over
the rational number field $\Q$;
\[
V_\Q:=\{x(u,v)=au^3+bu^2v+cuv^2+dv^3\mid a,b,c,d\in\Q\}.
\]
We express elements of $V_\Q$ as
$x=x(u,v)=au^3+bu^2v+cuv^2+dv^3$.
We identify $V_\Q$ with $\Q^4$ via
$V_\Q\ni x\mapsto(a,b,c,d)\in\Q^4$ and write as $x=(a,b,c,d)$ also.
Let $P(x)$ denote the discriminant of $x\in V_\Q$:
\[
P(x):={\rm Disc}(x(u,v))=b^2c^2+18abcd-4ac^3-4b^3d-27a^2d^2.
\]
The group $\spl_2(\Z)$ acts on $V_\Q$ by the linear change of
variables, and $P(x)$ is invariant under the action.

We recall the classification of $\spl_2(\Z)$-invariant lattices in $V_\Q$.
We put
\begin{align*}
L_1&:=\{x\in V_\Q\mid a,b,c,d\in\Z\}=\Z^4,\\
L_2&:=\{(a,b,c,d)\in L_1\mid a+b+d,a+c+d \in 2\Z \},\\
L_3&:=\{(a,b,c,d)\in L_1\mid a+b+c,b+c+d \in 2\Z \},\\
L_4&:=\{(a,b,c,d)\in L_1\mid b+c\in 2\Z \},\\
L_5&:=\{(a,b,c,d)\in L_1\mid a,d,b+c\in 2\Z \},
\end{align*}
and
\begin{align*}
L^\vee_1&:=\{x\in V_\Q\mid a,d\in\Z, b,c\in3\Z\},\\
L^\vee_2&:=L^\vee_1\cap L_3,
\quad
L^\vee_3:=L^\vee_1\cap L_2,
\quad
L^\vee_4:=L^\vee_1\cap L_5,
\quad
L^\vee_5:=L^\vee_1\cap L_4.
\end{align*}
We have $L_4\supset L_2$, $L_4\supset L_5$, $L_3\supset L_5$
and similar relations for $L^\vee_i$'s.
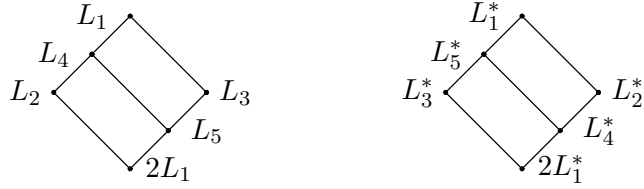
\begin{center}
\unitlength 0.1in
\begin{picture}( 11.7200,  8.8200)(  2.1500,-13.5700)
%
\special{pn 8}%
\special{pa 988 558}%
\special{pa 588 958}%
\special{pa 988 1358}%
\special{pa 1388 958}%
\special{pa 988 558}%
\special{pa 988 558}%
\special{pa 988 558}%
\special{fp}%
%
\special{pn 8}%
\special{pa 1188 1158}%
\special{pa 788 758}%
\special{pa 788 758}%
\special{pa 788 758}%
\special{fp}%
\put(7.9000,-5.6000){\makebox(0,0){$L_1$}}%
%
\special{pn 8}%
\special{sh 1}%
\special{ar 988 558 10 10 0  6.28318530717959E+0000}%
\special{sh 1}%
\special{ar 788 758 10 10 0  6.28318530717959E+0000}%
\special{sh 1}%
\special{ar 788 758 10 10 0  6.28318530717959E+0000}%
%
\special{pn 8}%
\special{sh 1}%
\special{ar 588 958 10 10 0  6.28318530717959E+0000}%
\special{sh 1}%
\special{ar 988 1358 10 10 0  6.28318530717959E+0000}%
\special{sh 1}%
\special{ar 1188 1158 10 10 0  6.28318530717959E+0000}%
\special{sh 1}%
\special{ar 1388 958 10 10 0  6.28318530717959E+0000}%
\special{sh 1}%
\special{ar 1388 958 10 10 0  6.28318530717959E+0000}%
\put(5.9000,-7.6000){\makebox(0,0){$L_4$}}%
\put(4.4000,-9.6000){\makebox(0,0){$L_2$}}%
\put(15.4000,-9.6000){\makebox(0,0){$L_3$}}%
\put(13.9000,-11.6000){\makebox(0,0){$L_5$}}%
\put(11.9000,-13.6000){\makebox(0,0){$2L_1$}}%
\end{picture}%

		\qquad\qquad
\unitlength 0.1in
\begin{picture}( 13.9500,  8.8200)( -1.2000,-13.5700)
%
\special{pn 8}%
\special{pa 876 558}%
\special{pa 476 958}%
\special{pa 876 1358}%
\special{pa 1276 958}%
\special{pa 876 558}%
\special{pa 876 558}%
\special{pa 876 558}%
\special{fp}%
%
\special{pn 8}%
\special{pa 1076 1158}%
\special{pa 676 758}%
\special{pa 676 758}%
\special{pa 676 758}%
\special{fp}%
\put(6.8000,-5.6000){\makebox(0,0){$L^\vee_1$}}%
%
\special{pn 8}%
\special{sh 1}%
\special{ar 876 558 10 10 0  6.28318530717959E+0000}%
\special{sh 1}%
\special{ar 676 758 10 10 0  6.28318530717959E+0000}%
\special{sh 1}%
\special{ar 676 758 10 10 0  6.28318530717959E+0000}%
%
\special{pn 8}%
\special{sh 1}%
\special{ar 476 958 10 10 0  6.28318530717959E+0000}%
\special{sh 1}%
\special{ar 876 1358 10 10 0  6.28318530717959E+0000}%
\special{sh 1}%
\special{ar 1076 1158 10 10 0  6.28318530717959E+0000}%
\special{sh 1}%
\special{ar 1276 958 10 10 0  6.28318530717959E+0000}%
\special{sh 1}%
\special{ar 1276 958 10 10 0  6.28318530717959E+0000}%
\put(4.8000,-7.6000){\makebox(0,0){$L^\vee_5$}}%
\put(3.3000,-9.6000){\makebox(0,0){$L^\vee_3$}}%
\put(14.3000,-9.6000){\makebox(0,0){$L^\vee_2$}}%
\put(12.8000,-11.6000){\makebox(0,0){$L^\vee_4$}}%
\put(10.8000,-13.6000){\makebox(0,0){$2L^\vee_1$}}%
\end{picture}%

\end{center}
In the previous paper \cite{sty}
the authors and Wakatsuki showed that up to $\Q^\times$-multiplication,
this is a complete list of $\spl_2(\Z)$-invariant lattices in $V_\Q$.
Hence there are $10$ different integral models of $V_\Q$.
The notation $L^\vee_i$ is because it is
isomorphic to the contragradient representation
${\rm Hom}(L_i,\Z)$ of $\spl_2(\Z)$.

The zeta functions are defined as follows.
\begin{defn}
For $1\leq i\leq 5$ and $1\leq j\leq2$, we define
\begin{align*}
\xi_{i,j}(s):=\sum_{\substack{x\in{\spl_2(\Z)}\backslash L_i\\(-1)^{j-1}P(x)>0}}\frac{|\stab(x)|^{-1}}{|P(x)|^s},
\qquad
\xi_{i,j}^\vee(s):=\sum_{\substack{x\in{\spl_2(\Z)}\backslash L_i^\vee\\(-1)^{j-1}P(x)>0}}
\frac{|\stab(x)|^{-1}}{|P(x)/27|^s},
\end{align*}
and call them the {\em zeta functions}
associated with $L_i$ or $L_i^\vee$.
Here $|\stab(x)|$ denote the number of stabilizers of $x$ in $\spl_2(\Z)$.
\end{defn}

Note that $|\stab(x)|$ is either $1$ or $3$ and that
for $x\in L^\vee_i$, $P(x)$ is a multiple of $27$.

Shintani \cite{shintania} showed that for $i=1$,
the standard integral models $L_1$ and $L_1^\vee$,
these zeta functions have holomorphic continuations
to the whole complex plane except for simple poles at $s=1,5/6$,
and satisfy a functional equation.
He also computed the residues explicitly.
In \cite{sty}, we proved similar analytic properties
for $2\leq i\leq 5$.

However, despite of Shintani's extensive study of
$\xi_{1,j}(s)$ and $\xi^\vee_{1,j}(s)$,
one other significant property was remain unrevealed
until 1990's.
The following identity
was conjectured by the first author
\cite{ohno} and proved by Nakagawa \cite{nakagawa}.
\begin{thm}[Conjectured in \cite{ohno}, proved in \cite{nakagawa}]
\label{thm:SNR_nakagawa}
We have
\begin{equation}\label{eq:SNR_nakagawa}
\xi^\vee_{1,1}(s)=\xi_{1,2}(s),
\qquad
\xi^\vee_{1,2}(s)=3\xi_{1,1}(s).
\end{equation}
\end{thm}
Although \eqref{eq:SNR_nakagawa} is quite simple,
no elementary proof of this theorem is known to the present.
In fact, Nakagawa proved them as a consequence
of the sophisticated use of class field theory.
As we will describe in Theorem \ref{thm:fe_SDFp},
this theorem has an
important application to the functional equation.
Hence it is natural to ask whether there exist
similar relations of the zeta functions for other integral models.
We will give the affirmative answer to this problem.

To state our results, we find it convenient to put
\[
\xi_i(s)
:=
\begin{pmatrix}\xi_{i,1}(s)\\\xi_{i,2}(s)\end{pmatrix},
\quad
\xi_i^\vee(s)
:=
\begin{pmatrix}\xi_{i,1}^\vee(s)\\\xi_{i,2}^\vee(s)\end{pmatrix},
\quad
A:=\begin{pmatrix} 0&1\\3&0\end{pmatrix}.
\]
Then \eqref{eq:SNR_nakagawa} is written as
$\xi_1^\vee(s)=A\cdot\xi_1(s)$.
For $i=2,3$,
the authors and Wakatsuki proved the following
in the previous paper \cite{sty}.
\begin{thm}[\cite{sty}]
\label{thm:SNR_sty}
We have
\begin{equation}\label{eq:SNR_sty}
\begin{split}
\xi_2^\vee(s)&=A\cdot\xi_2(s),\\
\xi_3^\vee(s)&=A\cdot\xi_3(s).
\end{split}
\end{equation}
\end{thm}
On the other hand, for $i=4,5$,
$\xi_i^\vee(s)$ and $A\cdot\xi_i(s)$ do not coincide.
These discrepancies themselves are not surprising since the
indices $[L_1:L_i]$ and $[L_1^\ast:L_i^\ast]$ do not coincide.
However, in view of \eqref{eq:SNR_nakagawa}, \eqref{eq:SNR_sty}
for $i=1,2,3$, one may believe that some corresponding
formulas should exist for $i=4,5$.

Indeed, we find such formulas in certain {\em linear combinations}
of the zeta functions. The following is a main result of this paper.
\begin{thm}[Main Theorem]\label{thm:SNR_ty}
We put
\begin{align*}
\theta(s)&
	:=\xi_1(s)-2\xi_3(s)-\xi_4(s)+4\xi_5(s),\\
\eta(s)&
	:=2^{2s}\left(\xi_4(s)-\xi_2(s)-\xi_5(s)+2^{1-4s}\xi_1(s)\right),\\
\theta^\vee(s)&
	:=2^{2s}\left(\xi_5^\vee(s)-\xi_3^\vee(s)-\xi_4^\vee(s)+2^{1-4s}\xi_1^\vee(s)\right),\\
\eta^\vee(s)&:=	\xi_1^\vee(s)-2\xi_2^\vee(s)-\xi_5^\vee(s)+4\xi_4^\vee(s).
\end{align*}
Then
\begin{equation}\label{eq:SNR_ty}
\begin{split}
\theta^\vee(s)&=	A\cdot\theta(s),\\
\eta^\vee(s)&=		A\cdot\eta(s).
\end{split}
\end{equation}
\end{thm}

We now give an application of these identities to the functional equations.
We put
\begin{align*}
\Delta_+(s)&:=
\left(\frac{2^43^3}{\pi^4}\right)^{s/2}
		\Gamma\left(\frac s2\right)\Gamma\left(\frac s2+\frac12\right)
		\Gamma\left(\frac{s}{2}-\frac1{12}\right)
		\Gamma\left(\frac{s}{2}+\frac1{12}\right),\\
\Delta_-(s)&:=
\left(\frac{2^43^3}{\pi^4}\right)^{s/2}
		\Gamma\left(\frac s2\right)\Gamma\left(\frac s2+\frac12\right)
		\Gamma\left(\frac{s}{2}+\frac5{12}\right)
		\Gamma\left(\frac{s}{2}+\frac7{12}\right).
\end{align*}
Then by plugging
\eqref{eq:SNR_nakagawa}, \eqref{eq:SNR_sty}
into the functional equations,
the followings were obtained:
\begin{thm}[\cite{ohno}, \cite{nakagawa}, \cite{sty}]\label{thm:fe_SDFp}
Let $i=1,2,3$.
For each sign, we put
\[
\xi_{i,\pm}(s):=\sqrt3\xi_{i,1}(s)\pm\xi_{i,2}(s).
\]
Let $a_1=0, a_2=a_3=2$.
Then they satisfy the functional equations
\begin{equation}\label{eq:fe_SDFp}
2^{a_is}\Delta_\pm(s)\xi_{i,\pm}(s)=
2^{a_i(1-s)}\Delta_\pm(1-s)\xi_{i,\pm}(1-s).
\end{equation}
\end{thm}
Similarly, as a consequence of \eqref{eq:SNR_ty},
we have the following.
\begin{thm}[Corollary to Theorem \ref{thm:SNR_ty}]\label{thm:fe_SDF}
We put
\[
\theta_\pm(s):=\sqrt3\theta_1(s)\pm\theta_2(s)
\quad
\text{and}
\quad
\eta_\pm(s):=\sqrt3\eta_1(s)\pm\eta_2(s),
\]
where
$\theta(s)=(\theta_1(s),\theta_2(s))$ and
$\eta(s)=(\eta_1(s),\eta_2(s))$.
Then
\begin{equation}\label{eq:fe_SDF}
\begin{split}
2^{s}\Delta_\pm(s)\theta_{\pm}(s)&=
2^{1-s}\Delta_\pm(1-s)\theta_{\pm}(1-s),\\
2^{s}\Delta_\pm(s)\eta_{\pm}(s)&=
2^{1-s}\Delta_\pm(1-s)\eta_{\pm}(1-s).
\end{split}
\end{equation}
\end{thm}
Hence the functional equations of the zeta functions
are expressed in self dual forms for all integral models.
In view of \eqref{eq:fe_SDFp} and \eqref{eq:fe_SDF},
we may say that the ``conductor'' of the Dirichlet series
$\xi_{1,\pm}(s)$, $\xi_{2,\pm}(s)$, $\xi_{3,\pm}(s)$, $\theta_\pm(s)$
and $\eta_\pm(s)$ are $2^43^3$, $2^83^3$, $2^83^3$, $2^63^3$ and $2^63^3$,
respectively\footnote{Either of these $10$ Dirichlet series are
of the form $\sum_{n\geq1}a_n/n^s$, and we can confirm that the greatest
common divisor of $\{n\mid a_n\neq0\}$ is $1$
by using the table of the coefficients in \cite{sty}.}.
We can describe the poles and residues of these Dirichlet series.
\begin{thm}\label{thm:residue}
The Dirichlet series
$\xi_{1,+}(s)$, $\xi_{2,+}(s)$, $\xi_{3,+}(s)$, $\theta_+(s)$
and $\eta_+(s)$ are holomprhic except for simple poles at $s=1$
and $s=5/6$, while
$\xi_{1,-}(s)$, $\xi_{2,-}(s)$, $\xi_{3,-}(s)$, $\theta_-(s)$
and $\eta_-(s)$ are holomorphic except for a simple pole at $s=1$.
The residues are given as follows:
\begin{gather*}
\begin{array}{l|ccccc}
\hline
\rule[-2mm]{0mm}{6mm}
&\xi_{1,+}(s)
&\xi_{2,+}(s)
&\xi_{3,+}(s)
&\theta_{+}(s)
&\eta_{+}(s)\\
\hline
\rule[-3mm]{0mm}{9mm}
\text{\rm Residue at $s=1$}
&\frac{2\sqrt3+3}{18}\pi^2
&\frac{2\sqrt3+3}{72}\pi^2
&\frac{2\sqrt3+3}{72}\pi^2
&\frac{7\sqrt3+9}{72}\pi^2
&\frac{5\sqrt3+9}{72}\pi^2\\
\rule[-3mm]{0mm}{8mm}
\text{\rm Residue at $s=5/6$}
&\frac{\Gamma(1/3)^3\zeta(1/3)}{3\pi}
&\frac{\Gamma(1/3)^3\zeta(1/3)}{12\pi}
&\frac{\Gamma(1/3)^3\zeta(1/3)}{12\pi}
&\frac{\Gamma(1/3)^3\zeta(1/3)}{3\sqrt[3]{2}\pi}
&\frac{\Gamma(1/3)^3\zeta(1/3)}{3\sqrt[3]{2}\pi}\\
\hline
\rule[-2mm]{0mm}{6mm}
\text{\rm ``Conductor''}
&2^43^3
&2^83^3
&2^83^3
&2^63^3
&2^63^3\\
\hline
\hline
\rule[-2mm]{0mm}{6mm}
&\xi_{1,-}(s)
&\xi_{2,-}(s)
&\xi_{3,-}(s)
&\theta_{-}(s)
&\eta_{-}(s)\\
\hline
\rule[-3mm]{0mm}{9mm}
\text{\rm Residue at $s=1$}
&\frac{2\sqrt3-3}{18}\pi^2
&\frac{2\sqrt3-3}{72}\pi^2
&\frac{2\sqrt3-3}{72}\pi^2
&\frac{7\sqrt3-9}{72}\pi^2
&\frac{5\sqrt3-9}{72}\pi^2\\
\hline
\rule[-2mm]{0mm}{6mm}
\text{\rm ``Conductor''}
&2^43^3
&2^83^3
&2^83^3
&2^63^3
&2^63^3\\
\hline
\end{array}
\end{gather*}
\end{thm}

It is an interesting phenomenon that the latter $5$
Dirichlet series are holomorphic at $s=5/6$.

Our basic approach to prove Theorem \ref{thm:SNR_ty} is
to reduce to Theorem \ref{thm:SNR_nakagawa},
as we did in the previous paper \cite{sty}
to prove Theorem \ref{thm:SNR_sty}.
However we need to argue more carefully since
the relations between $\xi_4(s),\xi_5(s)$ and $\xi_1(s)$
are not as direct as those of $\xi_2(s),\xi_3(s)$ and $\xi_1(s)$.
We look closely certain subsets of $L_1$ which are
no longer $\spl_2(\Z)$-invariant but invariant under
certain congruence subgroups such as $\Gamma_0(2)$ or $\Gamma(2)$,
and study them in terms of the {\em induction} in the category of $G$-sets.
The zeta functions behaves quite well with respect to this induction,
and these enables us to bring $\xi_4(s),\xi_5(s)$ and $\xi_1(s)$
into connection.

\bigskip

We note that for $i=1$, curious algebraic interpretations of
the set of integer orbits of $L_1$ and $L_1^\vee$ were known.
Precisely, ${\rm GL}_2(\Z)\backslash L_1$ has a canonical bijection
to the set of cubic rings, while $\spl_2(\Z)\backslash L_1^\vee$
essentially corresponds
to the set of $3$-torsions in ideal class groups of quadratic rings.
Indeed, these interpretations were key ingredient for
Nakagawa's proof of Theorem \ref{thm:SNR_nakagawa}
in terms of class field theory.
Such algebraic interpretations of integer orbits
for many other prehomogeneous vector spaces were discovered
rather systematically in Bhargava's surprising work of
{\em higher composition laws} \cite{bha}. In consideration of these results,
we expect that there might exist interesting interpretations
for integer orbits of $L_i, L_i^\vee$ for $2\leq i\leq 5$ also.
We hope the theory of integer orbits will be pursued further in the future.

This paper is organized as follows.
In Section \ref{sec:fiber}, we introduce the notion
of {\em induction}. After that we
study the set
\[
\{x\in L_1\mid P(x)\equiv l\mod 32\},
\quad
\{x\in L_1^\vee\mid P(x)/27\equiv -l\mod 32\},
\quad
l=4,20
\]
in some detail.
We prove in Proposition \ref{prop:FPE}
that actions of $\spl_2(\Z)$ to these sets are
induced from actions of $\Gamma_0(2)$ to their certain subsets.
The proof of Theorem \ref{thm:SNR_ty}
is given in Section \ref{sec:SNR}.
In Theorem \ref{thm:KSR}
we express partial zeta functions associated with the sets above
in terms of linear combinations of $\xi_i(s)$ or $\xi_i^\vee(s)$.
This enables us to reduce Theorem \ref{thm:SNR_ty}
to Theorem \ref{thm:SNR_nakagawa}.
In Section \ref{sec:fe_SDF}, we prove
Theorems \ref{thm:fe_SDF} and \ref{thm:residue}.

\bigskip

\noindent
{\bf Notations.}
The notations introduced above are used throughout this paper.
For a finite set $X$, we denote its cardinality by $|X|$.
If a group $G$ acts on a set $X$, then for $x\in X$
we put $G_x=\{g\in G\mid gx=x\}$.
In this paper we often consider congruence relations in $\Z$.
If $a-a'\in N\Z$ then we write $a\equiv a'(N)$ as well as
$a\equiv a' \mod N$. For
$a,a',a'',\dots\in\Z$,
``$a\mod N\equiv a'\equiv a''\equiv\cdots$''
means $a-a',a'-a'',\dots\in N\Z$.

The congruence subgroups of ${\rm SL}_2(\Z)$ are denoted by
\begin{gather*}
\Gamma(N)
=	\left\{
		\begin{pmatrix}p&q\\r&s\end{pmatrix}\in{\rm SL}_2(\Z)
	\ \vrule\ 
		\begin{pmatrix}p&q\\r&s\end{pmatrix}
		\equiv\begin{pmatrix}1&0\\0&1\end{pmatrix}
		\mod N
	\right\},\\
\Gamma_0(N)
=	\left\{
		\begin{pmatrix}p&q\\r&s\end{pmatrix}\in{\rm SL}_2(\Z)
	\ \vrule\ 
		r\equiv0\ (N)
	\right\},\quad
\Gamma^0(N)
=	\left\{
		\begin{pmatrix}p&q\\r&s\end{pmatrix}\in{\rm SL}_2(\Z)
	\ \vrule\ 
		q\equiv0\ (N)
	\right\}.
\end{gather*}
Hence $\Gamma(1)={\rm SL}_2(\Z)$.
Finally, we put $\EE:=2\Z$ and $\OO:=2\Z+1$,
the set of even integers and odd integers, respectively.

\section{Expressions in induced forms}
\label{sec:fiber}

To prove the main theorem, we use the notion of ``induction''
in the category of $G$-sets.
For the convenience of the reader, we summarize
its definition and basic properties.
The situation is quite similar to the induction of
representations of finite groups.
We omit the elementary proofs of
the basic facts.

Let $G$ be a group. Assume that
its subgroup $H$ acts on a set $Y$.
Then up to equivalence, there exists a unique pair
$(\iota,\widetilde Y)$ where $\widetilde Y$ is a $G$-set and
$\iota\colon Y\hookrightarrow \widetilde Y$ is an injective $H$-homomorphism
which satisfy the following conditions;
\begin{enumerate}
\item the map $\tilde\iota\colon H\backslash Y
\rightarrow G\backslash \widetilde Y$
induced from $\iota$ is bijective, and
\item for all $y\in Y$, $H_y=G_{\iota(y)}$.
\end{enumerate}
The pair $(\iota, \widetilde Y)$ is constructed as follows:
Consider an equivalence relation $\sim$ on $G\times Y$
so that
$(g,y)\sim (g',y')$ if and only if
there exists $h\in H$ such that $g'=gh^{-1}$ and $y'=hy$.
Let $\widetilde Y$ be the set of equivalence classes.
The equivalence class of $(g,y)$ is again
denoted by $(g,y)$. The well-defined map
$G\times \widetilde Y\ni (g',(g,x))\mapsto (g'g,x)\in \widetilde Y$
defines an action of $G$ on $\widetilde Y$.
Let $\iota\colon Y\ni y\mapsto (e,y)\in \widetilde Y$,
where $e\in G$ is the identity.
Then the pair $(\iota,\widetilde Y)$ satisfies the desired properties.
We denote this $\widetilde Y$ by $G\times_HY$.

Let $G$ be a group acting on a set $X$, and a subset $Y\subset X$
is invariant under the action of a subgroup $H\subset G$.
Then we can consider a natural map
$G\times_HY\ni (g,y)\mapsto gy\in X$ of $G$-sets.
When this map is bijective, we write $X=G\times_HY$
and say that {\em $(G,X)$ is induced from $(H,Y)$}.
We have the followings.
\begin{lem}\label{lem:fiber_equal}
 With the notations above,
$X=G\times_HY$ if and only if
\begin{enumerate}
\item the map $G\times Y\ni (g,y)\rightarrow gy\in X$ is surjective, and
\item for $y\in Y$ and $g\in G$, $gy\in Y$ if and only if $g\in H$.
\end{enumerate}
\end{lem}

\begin{lem}\label{lem:fiber_subset}
If $X'$ is a $G$-invariant subset of $G\times_HY$,
then $X'=G\times_H(Y\cap X')$.
\end{lem}

\begin{rem}\label{rem:induction}
In the category theoretic terminology, the correspondence
$\{\text{$H$-set}\}\ni Y\mapsto G\times_HY\in\{\text{$G$-set}\}$
is the left adjoint functor of the restriction functor
$\{\text{$G$-set}\}\ni X\mapsto X\in\{\text{$H$-set}\}$,
i.e.,
${\rm Hom}_H(Y,X)\cong{\rm Hom}_G(G\times_HY,X)$.
\end{rem}

We now consider the space of binary cubic forms.
For $l,N\in\Z$, we put
\[
L_1^{\equiv l(N)}
:=\{x\in L_1\mid P(x)\equiv l\ (N)\},
\qquad
(L_1^\vee)^{\cong l(N)}
:=\{x\in L_1^\vee\mid P(x)/27\equiv l\ (N)\}.
\]

The purpose of this section is to prove the following.
\begin{prop}\label{prop:FPE}
We put
\begin{align*}
X_1&=\{(a,b,c,d)\in\Z^4\mid b\in\OO, c\in2\OO, d\in 4\EE\},\\
X_2&=\{(a,b,c,d)\in\Z^4\mid b\in\OO, c\in2\OO, d\in 4\OO\},\\
X_3&=\{(a,b,c,d)\in\Z^4\mid a\in\OO, b,c\in\EE, d\in 2\OO\}
\end{align*}
and $X_i^\vee=X_i\cap L_1^\vee$ for $i=1,2,3$.
Then
\begin{align*}
&L_1^{\equiv4(32)}
	=\Gamma(1)\times_{\Gamma_0(2)}X_1,
&&L_1^{\equiv20(32)}
	=\Gamma(1)\times_{\Gamma_0(2)}(X_2\sqcup X_3),\\
&(L_1^\vee)^{\cong-20(32)}
	=\Gamma(1)\times_{\Gamma_0(2)}X_1^\vee,
&&(L_1^\vee)^{\cong-4(32)}
	=\Gamma(1)\times_{\Gamma_0(2)}(X_2^\vee\sqcup X_3^\vee).
\end{align*}
\end{prop}

We start with a lemma.
\begin{lem}\label{lem:4mod16}
For $x=(a,b,c,d)\in\Z^4$, $P(x)\equiv4\ (16)$ if and only if
one of the followings holds;
\begin{enumerate}
\item $a,b,c,d\in\OO, a+b+c+d\in2\OO$,
\item $b,c\in\EE, ad\in2\OO$,
\item $a\in 2\EE, b\in 2\OO, c\in\OO$,
\item $d\in 2\EE, c\in 2\OO, b\in\OO$,
\item $b+c\in\OO, a+c, b+d\in 2\EE$.
\end{enumerate}
\end{lem}
\begin{proof}
We write $P=P(x)$. We have
$P=(bc+ad)^2+4R+16(abcd-2a^2d^2)$ where $R=a^2d^2-ac^3-b^3d$.
Hence $P\mod 16\equiv(bc+ad)^2+4R$
and if $P\equiv4\ (16)$ then $bc+ad\in \EE$.
We note that in this case $(bc+ad)^2\equiv0\ (16)$ or $\equiv4\ (16)$
according as $bc+ad\in2\EE$ or $bc+ad\in2\OO$.
Also if $n\in\OO$ then $n^2\equiv 1\ (8)$ in general.

Assume $bc,ad\in\OO$. Then $a,b,c,d\in\OO$
and so $R\mod 4\equiv1-ac-bd\equiv1,3$.
Hence $P\equiv 4\ (16)$ if and only if
$bc+ad\in2\EE$ and $1-ac-bd\in 4\Z+1$.
Under the condition $a,b,c,d\in\OO$,
this is equivalent to $a+b+c+d\in2\OO$.
This is the case (1).

For the rest we consider the case $bc,ad\in\EE$.
In this case $R\mod 4\equiv-ac^3-b^3d$.
First assume $b,c\in\EE$.
Then $R\equiv0\ (4)$ and hence
$P\equiv4\ (16)$ if and only if $bc+ad\in2\OO$.
Hence $ad\in2\OO$ and we get the condition (2).
Next assume $b\in\EE$ and $c\in\OO$.
Since $R\equiv-ac\ (4)$,
$P\equiv4\ (16)$ if and only if
either (i) $ad+bc\in2\OO, ac\in 4\Z$
or (ii) $ad+bc\in2\EE, ac+1\in 4\Z$.
In the case (i), since $c\in\OO$, we have $a\in2\EE$
and so $ad+bc\in2\OO$ if and only if $b\in 2\OO$.
This is the case (3).
In the case (ii), since $a\in\OO$,
we have $d\in\EE$. Under the condition
$a,c\in\OO, b,d\in\EE$,
(ii) hold if and only if $a+c,b+d\in2\EE$.
Hence we get the condition
(A) : $b\in\EE, c\in\OO, a+c, b+d\in2\EE$.
Finally we assume $b\in\OO$ and $c\in\EE$.
By the same argument, $P\equiv 4\ (16)$ if and only if
either (4) or
(B) : $b\in\OO, c\in\EE, a+c,b+d\in2\EE$
is satisfied.
Since (A) or (B) is equivalent to the condition (5),
we have the lemma.
\end{proof}
\begin{lem}\label{lem:disjoint_union}
We put
\begin{align*}
X_1'&=	\{(a,b,c,d)\in\Z^4\mid c\in\OO, b\in 2\OO, a\in 4\EE\},\\
X_1''&=	\{(a,b,c,d)\in\Z^4\mid b+c\in\OO, a+c\in 2\EE, a+b+c+d\in 4\EE\},\\
X_2'&=	\{(a,b,c,d)\in\Z^4\mid c\in\OO, b\in 2\OO, a\in 4\OO\},\\
X_2''&=	\{(a,b,c,d)\in\Z^4\mid b+c\in\OO, a+c\in 2\EE, a+b+c+d\in 4\OO\},\\
X_3'&=	\{(a,b,c,d)\in\Z^4\mid d\in\OO, b,c\in \EE, a\in 2\OO\},\\
X_3''&=	\{(a,b,c,d)\in\Z^4\mid a,b,c,d\in\OO, a+b+c+d\in 2\OO\}.
\end{align*}
Then
\[
L_1^{\equiv4(32)}=X_1\sqcup X_1'\sqcup X_1'',\quad
L_1^{\equiv20(32)}=X_2\sqcup X_2'\sqcup X_2''\sqcup X_3\sqcup X_3'\sqcup X_3''.
\]
\end{lem}
\begin{proof}
If $P(x)\equiv4\ (16)$ then $P(x)\equiv 4\ (32)$ or $P(x)\equiv20\ (32)$.
Hence we can prove this lemma by examining each of five cases
listed in Lemma \ref{lem:4mod16}.
Since the argument is elemental and simple,
we briefly sketch the outline of the proof.

In case (1),
Since $R\mod 8\equiv a^2d^2-ac^3-b^3d\equiv 1-(ac+bd)$,
we have $P\mod 32\equiv (ad+bc)^2-4(ac+bd)+20$.
We note that $ad+bc,ac+bd\in2\EE$ (see the proof of
the previous lemma).
Moreover, $ad+bc+ac+bd=(a+b)(c+d)\in8\Z$ since
$a+b,c+d\in\EE$ and $a+b+c+d\in2\OO$.
Hence $P\equiv20\ (32)$. So this case corresponds to $X_3''$.

In case (2), since $ad+bc\in2\OO$, $(ad+bc)^2\equiv4\ (32)$.
Also $R\equiv4\ (8)$. Hence $P\equiv 20\ (32)$.
So this case corresponds to $X_3$ and $X_3'$.
In case (3), since $ad+bc\in2\OO$ also, $(ad+bc)^2\equiv4\ (32)$.
Moreover, $R\mod 8\equiv -ac^3\equiv -a$.
Hence $P\equiv 4-4a\ (32)$.
The case $a\in4\EE$ corresponds to $X_1'$ and
the case $a\in4\OO$ corresponds to $X_2'$.
In case (4), by the same argument we obtain
$X_1$ and $X_2$.

In case (5), first assume that $b\in\EE, c\in\OO$
Then $a\in\OO, d\in\EE$. We put $a+c=4m$ and $b+d=4n$
where $n,m\in\Z$.
Then since
\begin{align*}
ad+bc\mod 8&\equiv ad+(4m-a)(4n-d)\equiv 2ad-4an\equiv 2d-4n,\\
R\mod 8&\equiv d^2-ac\equiv 2d-a(4m-a)\equiv 2d+a^2-4am\equiv 2d-4m+1,
\end{align*}
we have
\[
P\mod 32\equiv 4(2d-4n)+4(2d-4m+1)\equiv 16(m+n)+4
\equiv4(a+b+c+d)+4.
\]
If we assume $b\in\OO, c\in\EE$, by the same argument
we have $P\mod 32\equiv 4(a+b+c+d)+4$.
The case $a+b+c+d\in4\EE$ corresponds to $X_1''$ and
the case $a+b+c+d\in4\OO$ corresponds to $X_2''$.
This completes the proof.
\end{proof}

We now give the proof of Proposition \ref{prop:FPE}.
\begin{proof}[Proof of Proposition \ref{prop:FPE}]
Let
$\tau=\begin{pmatrix}0&-1\\1&0\end{pmatrix}$
and
$\sigma=\begin{pmatrix}1&1\\-1&0\end{pmatrix}$.
Then $\{e,\tau,\sigma\}$ is a complete representative
of $\Gamma(1)/\Gamma_0(2)$.
Since
\[
\tau^{-1}(a,b,c,d)=(-d,c,-b,a),\quad
\sigma^{-1}(a,b,c,d)=(-d,c+3d,-b-2c-3d,a+b+c+d)
\]
by a simple computation we have $X_i'=\tau X_i$ and $X_i''=\sigma X_i$
for $i=1,2,3$.
Hence by Lemma \ref{lem:fiber_equal}
it is enough to show that for $\gamma\in\Gamma(1)$ and $x\in X_i$,
$x'=\gamma x\in X_i$ if and only if $\gamma\in\Gamma_0(2)$.
Let $x=(a,b,c,d)$, $x'=(a',b',c',d')$ and
$\gamma=\begin{pmatrix}p&q\\r&s\end{pmatrix}\in\Gamma(1)$.
Then
\[
\begin{pmatrix} a'\\b'\\c'\\d'\end{pmatrix}
=
\begin{pmatrix}
p^3		&p^2q		&pq^2		&q^3		\\
3p^2r		&p^2s+2pqr	&q^2r+2pqs	&3q^2s		\\
3pr^2		&qr^2+2prs	&ps^2+2qrs	&3qs^2		\\
r^3		&r^2s		&rs^2		&s^3		\\
\end{pmatrix}
\begin{pmatrix} a\\b\\c\\d\end{pmatrix}.
\]
We consider the case $i=1$. Let $x\in X_1$.
It is easy to see that if $\gamma\in\Gamma_0(2)$ then $x'\in X_1$.
Conversely, assume $x'\in X_1$. Then $b'\mod 2\equiv pra+psb$.
So $b'\in\OO$ implies $p\in\OO$. Hence $b'\mod 2\equiv ra+sb$,
$d'\mod 2\equiv ra+rsb$. So $b'-d'\in\OO$ implies $r\in\EE$,
namely $\gamma\in\Gamma_0(2)$. The cases $i=2,3$ are similarly proved.
Hence by Lemma \ref{lem:disjoint_union}, we obtain the first two formulas
of the proposition.
Since $L_1^{\equiv l(32)}\cap L_1^\vee=(L_1^\vee)^{\cong 3l(32)}$,
the rest two follow from
$L_1^{\equiv4(32)}\cap L_1^\vee=(L_1^\vee)^{\cong 12(32)}
=(L_1^\vee)^{\cong -20(32)}$,
$L_1^{\equiv20(32)}\cap L_1^\vee=(L_1^\vee)^{\cong 60(32)}
=(L_1^\vee)^{\cong -4(32)}$
and Lemma \ref{lem:fiber_subset}.
\end{proof}

\section{Proof of the main theorem}
\label{sec:SNR}

In this section we prove Theorem \ref{thm:SNR_ty}.
We start with a definition.

\begin{defn}\label{defn:partialzeta}
For a congruence subgroup $\Gamma$ of ${\rm SL}_2(\Z)$ and
a $\Gamma$-invariant subset $X$ of a lattice,
we define
\[
\xi_{j}(X,\Gamma,s):=\sum_{\substack{x\in{\Gamma}\backslash X\\(-1)^{j-1}P(x)>0}}
\frac{|\Gamma_x|^{-1}}{|P(x)|^s},
\qquad
\xi(X,\Gamma,s)
:=
\begin{pmatrix}\xi_1(X,\Gamma,s)\\\xi_2(X,\Gamma,s)\end{pmatrix}
\]
and call them {\em partial zeta functions} for the pair $(X,\Gamma)$.
\end{defn}
In this section the complex variable $s$ is always fixed
and so we mostly drop $s$ and write as $\xi(X,\Gamma)$.
If the expressions of $X,\Gamma$ contain parentheses
we may also write as $\xi[X,\Gamma]$.
By definition,
$\xi_i=\xi[L_i,\Gamma(1)]$ and $\xi_i^\vee=3^{3s}\xi[L_i^\vee,\Gamma(1)]$.
We define as follows.

\begin{defn}
We put
\[
\xi_1^{\equiv l(N)}:=\xi[L_1^{\equiv l(N)},\Gamma(1)],
\qquad
(\xi_1^\vee)^{\cong l(N)}:=\xi[(L_1^\vee)^{\cong l(N)},\Gamma(1)].
\]
\end{defn}
The crucial step of our proof of the main theorem is to express
$\xi_1^{\equiv l(32)}$ (resp. $(\xi_1^\vee)^{\cong -l(32)}$) for $l=4,20$
in terms of linear combinations of $\xi_i$'s (resp. $\xi_i^\vee$'s).
After we prepare necessary tools, we will do this
in Theorem \ref{thm:KSR} by using Proposition \ref{prop:FPE}.

To study the zeta functions, It will be convenient to consider
the following twisted action of $G_\Q:={\rm GL}_2(\Q)$ on $V_\Q$
which is compatible with the action of $\spl_2(\Z)$:
\[
(g\cdot x)(u,v)=\frac{1}{\det g}\cdot x(pu+rv, qu+sv),
\qquad
x\in V_\Q,
\quad
g=\begin{pmatrix}p&q\\r&s\\\end{pmatrix}\in G_\Q.
\]
Then $P(gx)=(\det g)^2P(x)$.
The followings are basic properties of the partial zeta functions.
\begin{prop}\label{prop:partialzeta}
The followings hold.
\begin{enumerate}
\item
If $X,X'$ are $\Gamma$-invariant and $X\cap X'=\emptyset$,
$\xi(X\sqcup X',\Gamma)=\xi(X,\Gamma)+\xi(X',\Gamma)$.
\item
If $X$ is $\Gamma$-invariant and $g\in G_\Q$,
$\xi(X,\Gamma)=(\det g)^{2s}\xi(gX,g\Gamma g^{-1})$.
\item
If $\Gamma'\subset \Gamma$ and $X$ is $\Gamma'$-invariant,
$\xi(X,\Gamma')=\xi(\Gamma\times_{\Gamma'}X,\Gamma)$.
\item
If $\Gamma'\subset \Gamma$ and $X$ is $\Gamma$-invariant,
$\xi(X,\Gamma')=[\Gamma:\Gamma']\xi(X,\Gamma)$.
\end{enumerate}
\end{prop}
\begin{proof}
Since (1), (2) and (3) immediately follow from the definition,
we consider (4).
Let $X^\pm=\{x\in X\mid \pm P(x)>0\}$. We have
\[
	\sum_{y\in\Gamma'\backslash X^\pm}
	\frac{|\Gamma'_y|^{-1}}{|P(y)|^s}
=	\sum_{x\in\Gamma\backslash X^\pm}
	\sum_{y\in \Gamma'\backslash \Gamma x}
	\frac{|\Gamma'_y|^{-1}}{|P(y)|^s}
=	\sum_{x\in\Gamma\backslash X^\pm}
	\frac{|\Gamma_x|^{-1}}{|P(x)|^s}
	\sum_{y\in \Gamma'\backslash \Gamma x}
\frac{|\Gamma_x|}{|\Gamma'_y|}.
\]
Hence (4) follows from the following result
in elementary group theory.
\end{proof}

\begin{lem}
Assume a group $G$ acts on a set $X$
and $H$ be an index finite subgroup of $G$.
Then for $x\in X$ with $|G_x|<\infty$,
\[
\sum_{y\in H\backslash Gx}\frac{|G_x|}{|H_y|}=[G:H],
\]
where in the summation of the left hand side, $y$ runs through
all the representatives of $H$-orbits in $Gx$.
\end{lem}
\begin{proof}
Consider the canonical bijections
$H\backslash Gx\simeq H\backslash(G/G_x)\simeq(H\backslash G)/G_x$.
If $y=gx, g\in G$, then
we have $|H_y|=|(g^{-1}Hg\cap G_x)|$
because $H_y=H\cap gG_xg^{-1}=g(g^{-1}Hg\cap G_x)g^{-1}$.
Since $g^{-1}Hg\cap G_x$ is the group of stabilizers
of $Hg\in H\backslash G$ in $G_x$, 
this implies that $|G_x|/|H_y|$ is equal to the cardinality of
the $G_x$-orbit of $Hg$ in $H\backslash G$.
Hence to sum up all the representative in the left hand side
is nothing but counting all the elements of the quotient set
$H\backslash G$ exactly one time for each.
\end{proof}

\begin{rem}
The formula in Proposition \ref{prop:partialzeta} (3)
indicates an advantage of using the induction.
The formula in (4) says that $\xi(X,\Gamma)$ is essentially
determined by $X$. In this sense we also say $\xi(X,\Gamma)$
as a partial zeta function for $X$, without referring to $\Gamma$.
\end{rem}

We consider partial zeta functions for
the quotient classes of $L_1$ by $2 L_1$.
Each class $(p,q,r,s)+2L_1$ is invariant under the action of $\Gamma(2)$.
\begin{defn}
For $p,q,r,s\in\{0,1\}$, we put
$\xi_{pqrs}:=\xi[(p,q,r,s)+2L_1,\Gamma(2)]$.
\end{defn}
If necessary, we also regard $p,q,r,s$ as elements of $\Z/2\Z$.
It is easy to see that
the number of $\Gamma(1)$-orbits of $L_1/2L_1$ is six.
By Proposition \ref{prop:partialzeta} (2),
this implies that there are six different
partial zeta functions $\xi_{pqrs}$.
More precisely, we can take six partial zeta functions
$\xi_{0000}, \xi_{0001}, \xi_{0010}, \xi_{0110}, \xi_{0111}, \xi_{1011}$
as representatives, and others are given by
\begin{align*}
\xi_{0001}=\xi_{1000}=\xi_{1111},
&&&
\xi_{0010}=\xi_{0100}=\xi_{0011}=\xi_{1100}=\xi_{0101}=\xi_{1010},\\
\xi_{0111}=\xi_{1110}=\xi_{1001},
&&&
\xi_{1011}=\xi_{1101}.
\end{align*}

The relations between
$\xi_i$'s and $\xi_{pqrs}$'s
are given as follows.

\begin{prop}\label{prop:mod2zeta}
\begin{enumerate}
\item
We have
\begin{gather*}
6\xi_1=\xi_{0000}+3\xi_{0001}+6\xi_{0010}+\xi_{0110}+3\xi_{0111}+2\xi_{1011},\\
6\xi_2=\xi_{0000}+3\xi_{0111},
\quad
6\xi_3=\xi_{0000}+\xi_{0110}+2\xi_{1011},\\
6\xi_4=\xi_{0000}+3\xi_{0001}+\xi_{0110}+3\xi_{0111},
\quad
6\xi_5=\xi_{0000}+\xi_{0110},
\quad
6\cdot 2^{-4s}\xi_1=\xi_{0000}.
\end{gather*}
\item
We have
\begin{gather*}
\xi_{0000}=6\cdot2^{-4s}\xi_1,
\quad
\xi_{0001}=2(\xi_4-\xi_2-\xi_5+2^{-4s}\xi_1),
\quad
\xi_{0010}=\xi_1+\xi_5-\xi_3-\xi_4,\\
\xi_{0110}=6(\xi_5-2^{-4s}\xi_1),
\quad
\xi_{0111}=2(\xi_2-2^{-4s}\xi_1),
\quad
\xi_{1011}=3(\xi_3-\xi_5).
\end{gather*}
\end{enumerate}
\end{prop}
\begin{proof}Since $[\Gamma(1),\Gamma(2)]=6$,
$6\xi_i
=[\Gamma(1),\Gamma(2)]\cdot\xi[L_i,\Gamma(1)]
=\xi[L_i,\Gamma(2)]$. Hence by dividing $L_i$
into the disjoint union of quotient classes modulo $2L_1$, we have (1).
For example,
\[
6\xi_3
=\sum_{\substack{p,q,r,s\in\Z/2\Z\\p+q+r=q+r+s=0}}\xi_{pqrs}
=\xi_{0000}+\xi_{0110}+\xi_{1011}+\xi_{1101}
=\xi_{0000}+\xi_{0110}+2\xi_{1011}.
\]
The formulas in (2) are easily obtained from (1).
\end{proof}

Now we will prove the following formulas.
The authors call the following relations as kaleidoscopic relations
in a joke but for meaning their sophisticated symmetry.
\begin{thm}\label{thm:KSR}
\begin{align*}
\xi_1^{\equiv4(32)}
	&=3\cdot2^{-2s}(\xi_5-2^{-4s}\xi_1),\\
(\xi_1^\vee)^{\cong-20(32)}
	&=3\cdot2^{-2s}(\xi_4^\vee-2^{-4s}\xi_1^\vee),\\
\xi_1^{\equiv20(32)}
	&=	(\xi_4-\xi_2-\xi_5+2^{1-4s}\xi_1)
		+2^{-4s}(\xi_1-\xi_4-2\xi_3+4\xi_5)
\\&\qquad	-2^{-2s}(\xi_1-\xi_3-2\xi_2+5\cdot2^{-4s}\xi_1),\\
(\xi_1^\vee)^{\cong-4(32)}
	&=	(\xi_5^\vee-\xi_3^\vee-\xi_4^\vee+2^{1-4s}\xi_1^\vee)
		+2^{-4s}(\xi_1^\vee-\xi_5^\vee-2\xi_2^\vee+4\xi_4^\vee)
\\&\qquad	-2^{-2s}(\xi_1^\vee-\xi_2^\vee-2\xi_3^\vee+5\cdot2^{-4s}\xi_1^\vee).
\end{align*}
\end{thm}
\begin{proof}
For subsets $A,B,C,D$ of $\Z$, we write
$(A,B,C,D)=\{(a,b,c,d)\in\Z^4\mid a\in A,b\in B,c\in C,d\in D\}$.
Let $g=\begin{pmatrix}1&0\\0&1/2\end{pmatrix}$.
Then by Proposition \ref{prop:partialzeta},
\begin{align*}
\xi[(A,B,2C,4D),\Gamma_0(2)]
&=(\det g)^{2s}\xi[g(A,B,2C,4D),g\Gamma_0(2)g^{-1}]\\
&=2^{-2s}\xi[(2A,B,C,D),\Gamma^0(2)].
\end{align*}
Hence by Propositions \ref{prop:FPE},
\ref{prop:partialzeta} and \ref{prop:mod2zeta},
\begin{align*}
\xi_1^{\equiv4(32)}
&=\xi[L_1^{\equiv4(32)},\Gamma]
=\xi[\Gamma\times_{\Gamma_0(2)}X_1,\Gamma]
=\xi[X_1,\Gamma_0(2)]\\
&=\xi[(\Z,2\Z+1,4\Z+2,8\Z),\Gamma_0(2)]
=2^{-2s}\xi[(2\Z,2\Z+1,2\Z+1,2\Z),\Gamma^0(2)]\\
&=2^{-2s}[\Gamma^0(2),\Gamma(2)]^{-1}\cdot\xi_{0110}
=3\cdot2^{-2s}(\xi_5-2^{-4s}\xi_1)
\end{align*}
and we have the first formula. Similarly,
the third formula follows from
\[
\xi_1^{\equiv20(32)}
=\xi[\Gamma\times_{\Gamma_0(2)}(X_2\sqcup X_3),\Gamma]
=\xi[X_2,\Gamma_0(2)]+\xi[X_3,\Gamma_0(2)]
\]
and
\begin{align*}
\xi[X_2,\Gamma_0(2)]
&=\xi[(\Z,2\Z+1,4\Z+2,8\Z+4),\Gamma_0(2)]\\
&=2^{-2s}\xi[(2\Z,2\Z+1,2\Z+1,2\Z+1),\Gamma^0(2)]\\
&=2^{-2s}[\Gamma^0(2):\Gamma(2)]^{-1}\cdot\xi_{0111}
=2^{-2s}(\xi_2-2^{-4s}\xi_1),\\
\xi[X_3,\Gamma_0(2)]
&=\xi[(2\Z+1,2\Z,2\Z,4\Z+2),{\Gamma_0(2)}]\\
&=\xi[(2\Z+1,2\Z,2\Z,2\Z),{\Gamma_0(2)}]
  -\xi[(2\Z+1,2\Z,2\Z,4\Z),{\Gamma_0(2)}]\\
&=2^{-1}\xi_{1000}
  -\xi[(\Z,2\Z,2\Z,4\Z),{\Gamma_0(2)}]
  +\xi[(2\Z,2\Z,2\Z,4\Z),{\Gamma_0(2)}]\\
&=2^{-1}\xi_{1000}
  -{2^{-2s}}\xi[(2\Z,2\Z,\Z,\Z),\Gamma^0(2)]
  +2^{-4s}\xi[(\Z,\Z,\Z,2\Z),{\Gamma_0(2)}]\\
&=2^{-1}\xi_{1000}
  -{2^{-1-2s}}{\textstyle\sum_{r,s=0,1}\xi_{00rs}}
  +2^{-1-4s}{\textstyle\sum_{p,q,r=0,1}\xi_{pqr0}}\\
&=2^{-1}\xi_{0001}
  -{2^{-1-2s}}(\xi_{0000}+\xi_{0001}+2\xi_{0010})\\
&\qquad  +2^{-1-4s}(\xi_{0000}+\xi_{0001}+4\xi_{0010}+\xi_{0110}+\xi_{0111})\\
&=(\xi_4-\xi_2-\xi_5+2^{-4s}\xi_1)
  -2^{-2s}(\xi_1-\xi_2-\xi_3+2^{2-4s}\xi_1)\\
&\qquad  +2^{-4s}(2\xi_1-2\xi_3-\xi_4+4\xi_5).
\end{align*}
By considering the intersection with $L_1^\vee$ of each subset in $L_1$,
the rest two formulas are proved similarly.
\end{proof}

\begin{rem}
In the previous paper \cite{sty}, we proved
$\xi_1^{\equiv5(8)}=\xi_2-2^{-4s}\xi_1$ and
$\xi_1^{\equiv1(8)}=\xi_3-2^{-4s}\xi_1$.
Hence with Theorem \ref{thm:KSR}, we have
\begin{align*}
\xi_2&=\xi_1^{\equiv5(8)}+2^{-4s}\xi_1,\\
\xi_3&=\xi_1^{\equiv1(8)}+2^{-4s}\xi_1,\\
(1-2^{-4s})\xi_4
&=\xi_1^{\equiv20(32)}+2^{-2s}(1-2^{-2s})(1+2^{1-4s})\xi_1\\
&\qquad+ (1-2^{1-2s})
	\left(
		\xi_1^{\equiv5(8)}-2^{-2s}\xi_1^{\equiv1(8)}
		+3^{-1}(1+2^{1-2s})2^{2s}\xi_1^{\equiv4(32)}
	\right),
\\
\xi_5&=3^{-1}2^{2s}\xi_1^{\equiv4(32)}+2^{-4s}\xi_1.
\end{align*}
Thus the coefficients of Dirichlet series
$\xi_2,\xi_3,\xi_4,\xi_5$ are expressed in terms of
those of $\xi_1$.
This is also valid for zeta functions for dual lattices.
\end{rem}

We are now ready to prove Theorem \ref{thm:SNR_ty}.
\begin{proof}[Proof of Theorem \ref{thm:SNR_ty}]
We first note that Nakagawa's formula $\xi_1^\vee=\xi_1A$
of Theorem \ref{thm:SNR_nakagawa}
implies $(\xi_1^\vee)^{\cong -l(N)}=A\xi_1^{\equiv l(N)}$.
The reason for the switch from $l\mod N$ to $-l \mod N$
is because $A$ replaces forms of positive discriminants
by forms of negative discriminants and vice-versa.
Recall that we have put
\begin{align*}
&\theta:=	\xi_1-\xi_4-2\xi_3+4\xi_5,
&&
\eta:=	2^{2s}\left(\xi_4-\xi_2-\xi_5+2^{1-4s}\xi_1\right),\\
&\theta^\vee:=	2^{2s}\left(\xi_5^\vee-\xi_3^\vee-\xi_4^\vee+2^{1-4s}\xi_1^\vee\right),
&&
\eta^\vee:=	\xi_1^\vee-\xi_5^\vee-2\xi_2^\vee+4\xi_4^\vee,
\end{align*}
and our goal is to get $\theta^\vee=\theta A$ and
$\eta^\vee=A\eta$.
By Theorems \ref{thm:SNR_nakagawa}, \ref{thm:SNR_sty}
and \ref{thm:KSR},
\begin{align*}
2^{2s}A\xi_1^{\equiv20(32)}
&=A\eta+2^{-2s}A\theta-A(\xi_1-\xi_3-2\xi_2+5\cdot2^{-4s}\xi_1)\\
&=A\eta+2^{-2s}A\theta-(\xi_1^\vee-\xi_3^\vee-2\xi_2^\vee+5\cdot2^{-4s}\xi_1^\vee).
\end{align*}
Therefore since $A\xi_1^{\equiv20(32)}=(\xi_1^\vee)^{\cong-20(32)}$,
we have
\[
A\eta+2^{-2s}A\theta
=\xi_1^\vee-2\xi_2^\vee-\xi_3^\vee+3\xi_4^\vee+2^{1-4s}\xi_1^\vee
=\eta^\vee+2^{-2s}\theta^\vee.
\]
Similarly, from
$A\xi_1^{\equiv4(32)}=(\xi_1^\vee)^{\cong-4(32)}$, we have
\[
\theta^\vee+2^{-2s}\eta^\vee
=A(\xi_1-\xi_2-2\xi_3+3\xi_5+2^{1-4x}\xi_1)
=A\theta+2^{-2s}A\eta.
\]
These two equalities are equivalent to 
$\theta^\vee=A\theta$ and 
$\eta^\vee=A\eta$.
\end{proof}

\section{Analytic properties of the zeta functions}
\label{sec:fe_SDF}

Now we will prove Theorems \ref{thm:fe_SDF} and \ref{thm:residue}.

\begin{proof}[Proof of Theorems \ref{thm:fe_SDF} and \ref{thm:residue}]
By \cite[Theorem 4.2]{sty}, we have the functional equation
\[
\xi_i(1-s)=2^{2a_is}{[L_1:L_i]}^{-1}M(s)\xi^\vee_i(s)
\]
where
\[
M(s):=
\frac{3^{3s-2}}{2\pi^{4s}}
\Gamma(s)^2\Gamma(s-\frac16)\Gamma(s+\frac16)
\begin{pmatrix}
\sin 2\pi s&\sin \pi s\\
3\sin \pi s&\sin2\pi s\\
\end{pmatrix}
\]
and $a_1=0, a_2=a_3=a_4=a_5=2$.
Hence
\begin{align*}
\theta(1-s)
&=\xi_1(1-s)-2\xi_3(1-s)-\xi_4(1-s)+4\xi_5(1-s)\\
&=2^{4s-1}M(s)\left(2^{1-4s}\xi_1^\vee(s)-\xi_3^\vee(s)-\xi_4^\vee(s)+\xi_5^\vee(s)\right)\\
&=2^{2s-1}M(s)\theta^\vee(s)\\
&=2^{2s-1}M(s)A\theta(s).
\end{align*}
Note that the last equality follows from Theorem \ref{thm:SNR_ty}.
Similarly, we have
\[
\eta(1-s)=2^{2s-1}M(s)A\eta(s).
\]
We put
\[
\Delta(s)=\begin{pmatrix}\Delta_+(s)&0\\0&\Delta_-(s)\\\end{pmatrix},
\quad
T=\begin{pmatrix}\sqrt3 &1\\ \sqrt3& -1\\\end{pmatrix}.
\]
Then, since $\Delta(1-s)TM(s)A=\Delta(s)T$ (this symmetrization
of $M(s)$ is due to Datskovsky and Wright \cite{dawra}),
we have
\begin{align*}
2^{1-s}\Delta(1-s)T\theta(1-s)&=2^s\Delta(s)T\theta(s),\\
2^{1-s}\Delta(1-s)T\eta(1-s)&=2^s\Delta(s)T\eta(s),
\end{align*}
and Theorem \ref{thm:fe_SDF} is proved.
Finally, Theorem 1.7 immidiately follows from the
residue formulas of $\xi_{i,j}(s)$ given in \cite[Theorem 4.2]{sty}.
We note that
\[
\frac{\sqrt[3]{2\pi}\Gamma(1/3)\zeta(2/3)}{3\Gamma(2/3)}
=\frac{\Gamma(1/3)^3\zeta(1/3)}{2\pi}.
\]
Interestingly, the residues at $s=5/6$ of
$\xi_{1,-}(s)$, $\theta_-(s)$, $\xi_{2,-}(s)$, $\xi_{3,-}(s)$
and $\eta_-(s)$ vanish.
This finishes the proof.
\end{proof}

\bigskip

\noindent
{\bf Acknowledgements.}
The authors would like to express their gratitude to
Tomoyoshi Ibukiyama and Don Zagier
for valuable discussion and encouragement.
The authors also thank Manjul Bhargava
for useful communications.
The authors are grateful to Noriyuki Abe for
providing a well constructed C++ program to produce experimental data.
The first author thanks the Max-Planck-Institut f\"{u}r Mathematik
in Bonn for its wonderful working condition and hospitality.
The first author is supported by JSPS Grant-in-Aid No.\!\! 20540033.
The second author is supported by JSPS Grant-in-Aid No.\!\! 20740018,
No.\!\!  20674001 and by JSPS Postdoctoral Fellowships for Research Abroad.

\end{document}